\documentclass[a4paper,reqno,11pt]{amsart}
\usepackage[T1]{fontenc}
\usepackage[utf8]{inputenc}
\usepackage{graphicx, color}
\usepackage[dvipsnames,svgnames,table]{xcolor}
\usepackage[foot]{amsaddr}
\usepackage[normalem]{ulem}

\usepackage{amssymb, amsmath, amsthm}

\makeatletter
\def\thm@space@setup{
\thm@preskip=4mm
\thm@postskip=0mm
}
\makeatother

\usepackage{enumerate, enumitem}

\setenumerate{label=\textup{(\roman*)}, noitemsep, topsep=3pt-\parskip, 
labelindent=.2em, leftmargin=*, widest=iii,}
\setitemize{noitemsep, topsep=-\parskip, labelindent=.2em, leftmargin=*, widest=iii,}

\usepackage{hyperref}
\usepackage[capitalise, compress, noabbrev, nameinlink]{cleveref}
\definecolor{linkblue}{named}{MidnightBlue}
\hypersetup{colorlinks=true, linkcolor=linkblue,  anchorcolor=linkblue,
	citecolor=linkblue, filecolor=linkblue, menucolor=linkblue,
	urlcolor=linkblue} 
\usepackage{mathtools}
\usepackage{thmtools, thm-restate}
\usepackage[longnamesfirst,numbers,sort&compress]{natbib}

\theoremstyle{plain}
\newtheorem{thm}{Theorem}
\newtheorem*{thm*}{Theorem}
\newtheorem{theorem}[thm]{Theorem}

\newtheorem{lemma}[thm]{Lemma}
\newtheorem*{lemma*}{Lemma}

\newtheorem*{cor*}{Corollary}

\newtheorem*{lem*}{Lemma}
\newtheorem{conjecture}[thm]{Conjecture}
\newtheorem*{conjecture*}{Conjecture}

\theoremstyle{remark}
\newtheorem{claim}[thm]{Claim}

\newenvironment{proofclaim}[1][]
    {\let\oldqed\qedsymbol\renewcommand{\qedsymbol}{\ensuremath{\lozenge}}\begin{proof}[Proof of claim] }{\end{proof}\renewcommand{\qedsymbol}{\oldqed}}

\newcommand{\N}{\mathbb{N}}

\let\leq\leqslant
\let\geq\geqslant

\let\subset\subseteq

\let\epsilon\varepsilon

\newcommand{\NN}{\mathbb{N}}

\DeclareMathOperator\pw{pw}
\DeclareMathOperator\tw{tw}

\definecolor{brightmaroon}{rgb}{0.76, 0.13, 0.28}
\newcommand{\defin}[1]{\emph{\textcolor{brightmaroon}{#1}}}

\newcommand{\C}{\mathcal{C}}

\title{Erd\H{o}s-P\'osa property of rooted tree minors}

\begin{document}

\author[Claus]{Quentin Claus}
\address[Q.~Claus]{D\'epartement de Mathématiques, Universit\'e libre de Bruxelles, Belgium}
\email{quentin.claus@ulb.be}

\author[Joret]{Gwena\"el Joret}
\address[G.~Joret]{D\'epartement d'Informatique, Universit\'e libre de Bruxelles, Belgium}
\email{gwenael.joret@ulb.be}

\author[Rambaud]{Clément Rambaud}
\address[C.~Rambaud]{Universit\'e Côte d'Azur, CNRS, Inria, I3S, Sophia Antipolis, France}
\email{clement.rambaud@normalesup.org}

\author[Robinson]{Eileen Robinson}
\address[E.~Robinson]{D\'epartement de Mathématiques, Universit\'e libre de Bruxelles, Belgium}
\email{eileen.robinson@ulb.be}

\thanks{Q.\ Claus and G.\ Joret are supported by the Belgian National Fund for Scientific Research (FNRS)}

\begin{abstract}
Fiorini, Joret, and Wood (2013) showed that tree minors satisfy the so-called Erd\H{o}s-P\'osa property with a linear bound: For every tree $T$ there exists a constant $c \geq 1$ such that, for every graph $G$ and integer $k\geq 0$, either $G$ contains $k$ vertex-disjoint subgraphs each containing a $T$-minor, or $G$ has a set $X$ of at most $c k$ vertices such that $G-X$ has no $T$-minor.  
In this paper, we prove that the same result remains true if, given a subset $S$ of vertices of $G$, one only considers $T$-minors of $G$ that are rooted in $S$. Here, a $T$-minor is rooted in $S$ if there is a minor-model of $T$ where each branch set contains a vertex from $S$.  

This result can be seen as a generalization of the classical $S$-Path Theorem of Gallai, which corresponds to the case $T=K_2$.  
The upper bound on the size of $X$ is best possible up to the value of the constant $c$, and improves on an earlier $O(k^2)$ bound due to Hodor, La, Micek, and Rambaud (2026).  
\end{abstract}

\maketitle

\section{Introduction}

The well-known theorem of \citet{EP1965} states that, for every graph $G$ and positive integer $k$, either $G$ has $k$ vertex-disjoint cycles, or $G$ has a set $X$ of $O(k \log k)$ vertices meetings all cycles of $G$. 
This was generalized by \citet{RS1986} as follows: For every fixed planar graph $H$ there is a function $f(k)$ such that, for every graph $G$ and positive integer $k$, either $G$ has $k$ vertex-disjoint subgraphs each containing $H$ as a minor, 
or $G$ has a set $X$ of at most $f(k)$ vertices such that $G-X$ does not $H$ as a minor. 
(The setup of \citet{EP1965} corresponds to the case $H=K_3$.) 

In the above result of Robertson and Seymour, there is a $\Omega(k \log k)$ lower bound on the function $f(k)$ if $H$ contains a cycle, and it has been shown~\cite{CvBHJR2019} that this is the right order of magnitude for $f(k)$ in this case. 
If $H$ is a forest, no such lower bound holds, and \citet{FJW2013} showed that the result holds with $f(k)\in O(k)$ in that case. 
Recently, \citet{DJMM25} obtained a tight bound when $H$ is a tree on $t$ vertices, namely, $f(k)=t(k-1)$. This is best possible, as shown by a complete graph on $tk-1$ vertices. 

In this paper, we revisit this topic from the point of view of {\em rooted minors}. 
Given a graph $G$ and a subset $S\subseteq V(G)$ of vertices, we say that $G$ contains an \defin{$S$-rooted} model of a graph $H$ if there is a model of $H$ in $G$ where each branch set contains a vertex of $S$.
Here, by a \defin{model} of a graph $H$ in a graph $G$ 
we mean a family $(B_x \mid x \in V(H))$ of pairwise disjoint nonempty subsets of $V(G)$ called \defin{branch sets},
each inducing a connected subgraph of $G$,
such that for every $xy \in E(H)$, 
there is an edge in $G$ with one endpoint in $B_x$ and one endpoint in $B_y$.

\citet{hodor2024quickly} showed that, for every fixed tree $T$, the set of $S$-rooted models of $T$ satisfy the Erd\H{o}s-P\'osa property in the following sense: 
For every graph $G$, set $S\subseteq V(G)$, and positive integer $k$, 
either $G$ contains $k$ vertex-disjoint subgraphs each containing an $S$-rooted model of $T$, or 
$G$ has a set $X$ of $O(k^2)$ vertices such that $G-X$ has no $S$-rooted model of $T$. 

In this paper, we show that the $O(k^2)$ bound on the size of $X$ can be reduced to a $O(k)$ bound. 

\begin{theorem}[Main result]
    \label{thm:main}
    There exists a function $g\colon \N \to \N$ such that, 
    for every $t$-vertex tree $T$, every graph $G$, every set $S\subseteq V(G)$, and every positive integer $k$, 
    either $G$ contains $k$ vertex-disjoint $S$-rooted model of $T$, or 
    $G$ has a set $X$ of at most $g(t)\cdot k$ vertices such that $G-X$ has no $S$-rooted model of $T$. 
\end{theorem}

The upper bound on the size of $X$ in the above result is best possible as a function of $k$.  
On the other hand, the constant factor $g(t)$ appearing in our proof is a fast growing function of $t$ and is most likely not optimal. 

We remark that, when restricting to the unrooted setting (i.e.\ when $S=V(G)$), our proof is genuinely different from the two previous proofs~\cite{FJW2013, DJMM25} giving a $O(k)$ bound. 
The idea of the proof in~\cite{FJW2013} is to show that either $G$ contains a model of the desired tree $T$ that has {\em constant size} (in which case we remove all vertices in the model and induct on $k$), or that a reduction operation can be applied to $G$, producing a  smaller graph that is `equivalent' to $G$ (in which case we are done by induction on $|V(G)|$). 
The proof in~\cite{DJMM25} works by applying a lemma of~\citet{D95}, which gives a model of $T$ in $G$ together with a small-order separation that allows to cleanly cut off the model from the rest of the graph. 
We tried to extend each of these two approaches to the rooted setting without success, which is why we developed a new approach. 
Our approach works by induction on $t$. 
While it does not give a bound as good as in~\cite{DJMM25}, our proof is conceptually simple, and thus we believe it is of interest in the unrooted setting as well. 

Additionally, we generalize \Cref{thm:main} to the case of rooted models of a forest instead of a tree, and we also allow looking for models that are `partially rooted' instead of `fully rooted'.  
To state the result, we need to extend the notion of models
to pairs $(G,S)$ for $G$ a graph and $S \subseteq V(G)$.

For all graphs $G,H$ and all $S \subseteq V(G), R \subseteq V(H)$,
a \defin{model of $(H,R)$ in $(G,S)$} is a model $(B_x \mid x \in V(H))$ of $H$ in $G$ 
such that $B_x \cap S \neq \emptyset$ for every $x \in R$. 
Observe that the case $R=V(H)$ corresponds to an $S$-rooted model of $H$, while 
$R=\emptyset$ corresponds to a standard (unrooted) model of $H$; thus, this definition allows to interpolate between these two notions. 
Using techniques from \citet{DJMM25} on \Cref{thm:main},
we obtain the following generalization.

\begin{theorem}\label{thm:main_arbitrarily_rooted_forest}
    There is a function $g' \colon \NN \to \NN$ such that the following holds. 
    For every $t$-vertex forest $F$ and every $R \subseteq V(F)$,
    for every graph $G$ and every $S \subseteq V(G)$,
    for every positive integer $k$, 
    either
    there are $k$ vertex-disjoint models of $(F,R)$ in $(G,S)$, or
    there is a set $X \subseteq V(G)$ of size at most $g'(t)\cdot k$ such that
    there is no model of $(F,R)$ in $(G-X, S \setminus X)$. 
\end{theorem}

A generalization of rooted minors was recently proposed by \citet{colorful_minors}
under the name of $q$-colorful minors:
Every vertex is given a set of colors from a pool of size $q$,
and the minor relation is modified to respect these colors. 
For $q=1$, this corresponds to the notion of rooted minors appearing in \Cref{thm:main_arbitrarily_rooted_forest}, where the sets $R \subseteq V(F)$ and 
$S \subseteq V(G)$ describe which vertices of $F$ and $G$, respectively, receive the unique color that is available.
\citet{colorful_minors} characterized the $q$-colored graphs whose models satisfy 
the Erd\H{o}s-P\'osa property, for every $q \in \NN$. 
In particular, the property fails to hold for some (very simple) $2$-colored trees, thus there is no hope of generalizing \Cref{thm:main_arbitrarily_rooted_forest} to this setting.

\subsubsection*{Organization of the paper}
The paper is organized as follows. 
We introduce the necessary definitions and preliminary lemmas in \cref{sec:prelim}. 
Then, we prove \cref{thm:main} in \cref{sec:proof}, and \cref{thm:main_arbitrarily_rooted_forest} in \cref{sec:forests}.  
In \cref{sec:paths_and_stars}, we revisit \cref{thm:main} and show better bounds on the size of the hitting set $X$ for some specific trees $T$, namely, paths and stars.  
Motivated by these improved bounds, we conclude in \cref{sec:open_problems} with a conjecture about the best possible bound in \cref{thm:main}. 

\section{Preliminaries}
\label{sec:prelim}
For $n\in \mathbb{N}$, let $[n]$ denote the set $\{1,2,\ldots, n\}$. 

All graphs in this paper are finite, simple, and undirected. 
Given a graph $G$, we denote by $V(G)$ its vertex set, and by $E(G)$ its edge set.

The \defin{length} of a path $P$ in a graph $G$ is the number of edges in $P$.  
Given two sets $A,B \subseteq V(G)$, a path $P$ in $G$ is called an \defin{$A$--$B$ path} if its vertices can be enumerated as $v_1,\ldots,v_k$ along $P$ with $V(P)\cap A = \{v_1\}$ and $V(P)\cap B = \{v_k\}$.  In case $A$ contains a single vertex $v$, we sometimes simply write `$v$--$B$ path' instead of `$\{v\}$--$B$ path', and same for $B$.

For a tree $T$, we denote by $L(T)$ the set of all the leaves of $T$, 
where a \defin{leaf} of $T$ is a vertex of $T$ with degree exactly $1$. 
A \defin{cubic tree} is a tree where every vertex has degree $1$ or $3$,
and a \defin{subcubic tree} is a tree where every vertex has degree at most $3$.
A \defin{rooted tree} is a tree where a vertex is specified to be the root.  The \defin{height} of a rooted tree is the maximum distance between the root and a vertex of the tree. 
For $h \geq 0$, the \defin{complete ternary tree} of height $h$ is the unique rooted cubic tree of height $h$ where every leaf is at distance $h$ from the root.   
If one removes a neighbor of the root together with all its subtree, the resulting rooted tree is called the \defin{complete binary tree} of height $h$. 

Observe that the complete ternary tree of height $h$ has $3 \cdot 2^h-2$ vertices and the complete binary tree of height $h$ has $2^{h+1}-1$ vertices.

For a graph $G$ and a set $U\subseteq V(G)$, we define the \defin{neighborhood} of $U$ in $G$ as the set of vertices in $V(G) \setminus U$ that are adjacent to at least one vertex in $U$, and we denote it by $N_G(U)$.  

For a graph $G$ and a set $U\subseteq V(G)$, we denote by $G-U$ the subgraph of $G$ obtained by removing all the vertices in $U$.  
If $U$ is a singleton $\{u\}$, we simply write $G-u$ instead of $G-\{u\}$.

A \defin{tree decomposition} of a graph $G$ is a pair $\mathcal D = \big(T,(W_x \mid x \in V(T))\big)$ where $T$ is a tree and $(W_x \mid x \in V(T))$ is a collection of subsets of $V(G)$ called the \defin{bags} of $\mathcal{D}$, and such that 
\begin{enumerate}
    \item for each vertex $v\in V(G)$, the set $\{x \in V(T) \mid v\in W_x\}$ induces a nonempty subtree of $T$, and 
    \item for each edge $vw\in E(G)$, there exists $x\in V(T)$ such that both $v$ and $w$ are in the bag $W_x$.
\end{enumerate}
The \defin{width} of the tree decomposition $\mathcal{D}$ is $\max\{|W_x| \mid x\in V(T)\}-1$. 
The \defin{treewidth} of $G$, denoted by $\tw(G)$, is the minimum width of a tree decomposition of $G$.  

A \defin{path decomposition} of a graph $G$ is a tree decomposition $\big(T,(W_x \mid x \in V(T))\big)$ of $G$ where $T$ is a path,
and the \defin{pathwidth} of a graph $G$, denoted by $\pw(G)$, 
is the minimum width of a path decomposition of $G$.
It will be convenient to denote path decompositions simply as a sequence $(W_1, W_2, \dots, W_m)$ of bags ordered according to the path $T$.  

Let $G$ be a graph. 
A \defin{separation} of $G$ is a pair $(A,B)$ of subsets of $V(G)$ such that
$A \cup B = V(G)$ and
there is no edge in $G$ with one endpoint in $A \setminus B$ and one endpoint in $B \setminus A$.
The \defin{order} of this separation is $|A \cap B|$. 

Let $G$ and $H$ be graphs and let $S \subseteq V(G)$.  
Recall that a model of $H$ in $G$ is said to be  
$\mathit{S}$-rooted
if the branch set of each vertex of $H$ contains at least one vertex of $S$.  
To keep notations light when considering subgraphs, it will be convenient to also allow $S$ to contain vertices not in $G$ in this definition, 
and so ``$S$-rooted models of $H$ in $G$''
refers to ``$(S \cap V(G))$-rooted models of $H$ in $G$''. 
Similarly, when considering models of $(H,R)$ in $(G,S)$ in the paper, where $R\subseteq V(H)$, we allow $S$ to contain vertices not in $G$. 

For all graphs $G,H$ and all $S \subseteq V(G), R \subseteq V(H)$,
a \defin{model of $(H,R)$ in $(G,S)$} is a model $(B_x \mid x \in V(H))$ of $H$ in $G$ 
such that $B_x \cap S \neq \emptyset$ for every $x \in R$. 

The following definitions where introduced in \cite{hodor2024quickly,JansenSwennnenhuis2024} (see also \cite{claus2026excludingapexforestfanquickly, BlowupStructureExcludingTreeOrApexTree} for more background).  
Given a graph $G$ and $S \subseteq V(G)$, a \defin{tree decomposition of} $\mathit{(G,S)}$ is a pair $\big(T,(W_x \mid x \in V(T))\big)$ such that
\begin{enumerate}
    \item $\big(T,(W_x \mid x \in V(T))\big)$ is a tree decomposition of $G[\bigcup_{x\in V(T)}W_x]$,
    \item $S \subseteq \bigcup_{x \in V(T)} W_x$, and
    \item for every connected component $C$ of $G-\bigcup_{x \in V(T)} W_x$, there exists $x \in V(T)$ such that $N_G(V(C)) \subseteq W_x$.
\end{enumerate}
Again, the width of this tree decomposition is $\max \{|W_x| \mid x \in V(T)\}-1$, and
the \defin{treewidth of} $\mathit{(G, S)}$, denoted by $\tw(G, S)$, is the minimum width of a tree decomposition of $(G, S)$. 
(As for rooted models, to keep notations light in the proofs it will be convenient to allow $S$ to contain vertices not in $G$ in these definitions, and thus we do so.) 
A \defin{path decomposition of} $\mathit{(G,S)}$ is a tree decomposition $\big(T,(W_x \mid x \in V(T))\big)$ of $(G,S)$
where $T$ is a path, and the 
\defin{pathwidth of} $\mathit{(G, S)}$, denoted by $\pw(G, S)$, is the minimum width of a path decomposition of $(G, S)$. 
Note that for every graph $G$ and every $S \subseteq V(G)$, $\tw(G,S) \leq \pw(G,S)$,
and $\tw(G,V(G)) = \tw(G)$, $\pw(G,V(G)) = \pw(G)$. 
\citet{hodor2024quickly} proved the following theorem, which we will use in our main proof.   
It can be seen as an analogue of the classical Excluded Tree Minor Theorem of \citet{RS83} (see also \citet{D95}).
  
 \begin{theorem}[Theorem 7, \citet{hodor2024quickly}]
    For every forest $F$ with at least one vertex, for every graph $G$ and for every $S\subseteq V(G)$, if $G$ has no $S$-rooted model of $F$, then $\pw(G, S)\leq 2|V(F)|-2$.
    \label{diestelSrootedoriginal}
\end{theorem}

The following lemma will be useful in multiple parts of the paper. 

\begin{lemma}[Corollary of Lemma~23 in \cite{hodor2024quickly}] \label{lem:erdos_posa_pw}
    Let $G$ and $H$ be graphs such that $H$ is connected, let $S\subseteq V(G)$, and let $k, w$ be natural numbers such that $\tw(G,S)<w$. 
    If $G$ does not contain $k$ vertex-disjoint $S$-rooted models of $H$, then, 
    there exists a set $X\subseteq V(G)$, such that $|X|\leq w(k-1)$, 
    and $G-X$ does not contain any $S$-rooted model of $H$.
\end{lemma}

Next, we introduce a key object, which is a variant of a definition appearing in~\cite{GJNW23}\footnote{For readers familiar with~\cite{GJNW23}, the main change here is the addition of roots.}.
Given a graph $G$ and a set $S\subseteq V(G)$, we define a sequence $\C_0(G, S), \C_1(G, S), \C_2(G, S), \dots$ of sets of subgraphs of $G$, inductively as follows. 
The set $\C_0(G, S)$ is the set of all single-vertex subgraphs of $G$ whose unique vertex belongs to $S$. 
For $t\geq 1$, the set $\C_t(G, S)$ is the set of all connected subgraphs $H$ of $G$ that have the following form:
\begin{enumerate}[label=(C\arabic*)]
    \item\label{item:def_Ct:star} $H$ is the union of three vertex-disjoint subgraphs $H_1, H_2, H_3 \in \C_{t-1}(G, S)$ and three paths $P_1, P_2, P_3$ such that the three paths share a common endpoint $v$ and are otherwise vertex-disjoint, the vertex $v$ is outside $H_1, H_2, H_3$, and for every $i\in [3]$, the path $P_i$ has its other endpoint in $V(H_i)$ and is otherwise vertex-disjoint from $H_1, H_2, H_3$; or
    \item\label{item:def_Ct:triangle} $H$ is the union of three vertex-disjoint subgraphs $H_1, H_2, H_3 \in \C_{t-1}(G, S)$ and three paths $P_1, P_2, P_3$ such that 
    the interiors of $P_1,P_2,P_3$ are pairwise vertex-disjoint and vertex-disjoint from $V(H_1) \cup V(H_2) \cup V(H_3)$,
    and for each $i \in [3]$,
    $P_i$ has one endpoint in $H_{i+1}$ and the other in $H_{i+2}$ (where indices are taken cyclically). 
\end{enumerate}
See \Cref{fig:def_Ct} for an illustration of this definition.
\begin{figure}
    \centering
    \includegraphics{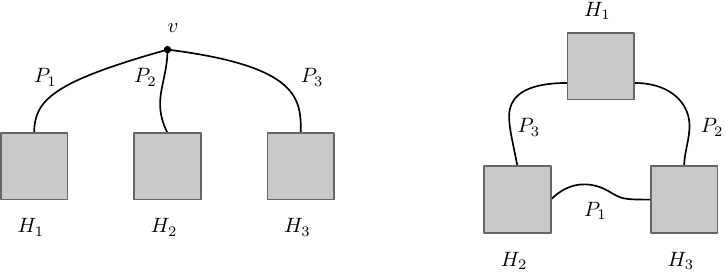}
    \caption{The two cases \ref{item:def_Ct:star} (on the left) 
        and \ref{item:def_Ct:triangle} (on the right)
        in the definition of $\mathcal{C}_t(G,S)$.}
    \label{fig:def_Ct}
\end{figure}
The following lemma is adapted from Lemmas~2.1 and~2.2 in~\cite{GJNW23}, we give its proof for completeness. 
\begin{lemma}\label{lemma:properties_of_Ct}
    Let $G$ be a graph and let $S\subseteq V(G)$. 
    Then, for every nonnegative integer $t$, the following holds for every $H\in \C_t(G, S)$:
    \begin{enumerate}[label={\normalfont(\alph*)}]
        \item $H$ is connected and has maximum degree at most $2t+1$; \label{prop:max_degree} 
        \item $H$ contains an $S$-rooted model of a complete binary tree of height $t$; \label{prop:CBT} 
        \item $H$ contains an $S$-rooted model of every $t$-vertex tree. \label{prop:every Ct_contains_every_tree}
    \end{enumerate}
\end{lemma}

\begin{proof} 
    It is clear from the definition that every graph in $\C_t(G, S)$ is connected. 
    To prove \ref{prop:max_degree}, it remains to show the upper bound of $2t+1$ on the maximum degree. 
    The proof is by induction on $t$: 
    For $t=0$, graphs in $\C_t(G, S)$ have maximum degree $0$, 
    for $t=1$ they have maximum degree $3=2t+1$, and  
    for $t\geq 2$, they have maximum degree at most two more than those in $\C_{t-1}(G, S)$. 
    
    In order to show \ref{prop:CBT}, we prove the following slightly stronger statement, which helps the induction go through: 
    \begin{enumerate}[label={\normalfont(\alph*')}]
        \setcounter{enumi}{1}
        \item  For every nonnegative integer $t$, every $H\in \C_t(G, S)$, and every $u\in V(H)$, there is an 
        $S$-rooted model of a complete binary tree of height $t$ in $H$ such that the root branch set contains $u$. \label{prop:CBT'}
    \end{enumerate}
    
    For $t=0$, this is clear, so we assume $t \geq 1$ and that the result holds for $t-1$.
    Let $H \in \mathcal{C}_t(G,S)$.
    First suppose we are in the case \ref{item:def_Ct:star} of the definition of $\mathcal{C}_t(G,S)$, 
    and let $H_1, H_2, H_3, P_1, P_2, P_3, v$ be as in this definition.
    Without loss of generality, $u \in V(H_1) \cup V(P_1)$.
    Then, let $u_2, u_3$ be the endpoints of respectively $P_2, P_3$ distinct from $v$.
    By the induction hypothesis, $H_2$ (resp.\ $H_3$) contains an $S$-rooted model $\mathcal{M}_2$ (resp.\ $\mathcal{M}_3$)
    of a complete binary tree of height $t-1$ such that $u_2$ (resp. $u_3$) is
    in the branch set of the root.
    Then, the union of $\mathcal{M}_2$ and $\mathcal{M}_3$ together with 
    the new branch set $B = V(H_1) \cup V(P_1) \cup (V(P_2) \setminus \{u_2\}) \cup (V(P_3) \setminus \{u_3\})$
    yields an $S$-rooted model of the complete binary tree of height $t$ with $u$ in the
    branch set $B$ of the root.
    
    Now suppose we are in the case \ref{item:def_Ct:triangle} of the definition of $\mathcal{C}_t(G,S)$, 
    and let $H_1, H_2, H_3, P_1, P_2, P_3$ be as in this definition.
    Without loss of generality, $u \in V(H_1) \cup (V(P_2) \setminus V(H_3))$.
    Let $u_2$ be the endpoint of $P_3$ in $H_2$,
    and let $u_3$ be the endpoint of $P_2$ in $H_3$.
    Then, by the induction hypothesis,
    $H_2$ (resp.\ $H_3$) contains an $S$-rooted model $\mathcal{M}_2$ (resp.\ $\mathcal{M}_3$)
    of a complete binary tree of height $t-1$ such that $u_2$ (resp. $u_3$) is
    in the branch set of the root.
    Then, the union of $\mathcal{M}_1, \mathcal{M}_2$ with the new branch set
    $B = V(H_1) \cup (V(P_2) \setminus V(H_3)) \cup V(P_3)\setminus V(H_2)$ yields a model of a complete binary
    tree of height $t$ with $u$ lying in the root branch set $B$.
    This proves \ref{prop:CBT'}.
    
    Finally, to deduce \ref{prop:every Ct_contains_every_tree},
    it is enough to show that every tree on $t$ vertices is a minor of
    the complete binary tree $B_t$ of height $t$.
    This is a well-known observation, but we include a proof for completeness.
    Actually, we show that for every tree $T$ on $t$ vertices rooted at a vertex $r$,
    there is a model of $T$ in $B_t$ whose root branch set contains the root of $B_t$.
    For $t=1$, this is clear, so suppose $t \geq 2$.
    Let $T_1, \dots, T_c$ be the connected components of $T-r$ with $|V(T_1)| \geq \dots \geq |V(T_c)|$,
    and let $H_1, \dots, H_{t-1}$ be the connected components of the forest obtained from $B_t$
    by removing the vertex set of a root-to-leaf path $Q$.
    Relabeling if necessary, we may assume that $H_i$ is isomorphic to $B_i$ for every $i \in [t-1]$,
    with the root having a neighbor in $V(Q)$ in $B_t$.
    Therefore, by the induction hypothesis, for every $i \in [c]$,
    since $|V(T_i)| \leq t-i$, there is a model $\mathcal{M}_i$ of $T_i$ in $H_{t-i}$
    whose root branch set contains a neighbor of $V(Q)$ in $B_t$.
    It follows that the union of these models
    $\mathcal{M}_{1}, \dots, \mathcal{M}_{c}$
    with the root branch set $V(Q)$ gives a model of $T$ in $B_t$, with the root of $B_t$ in the root branch set, as desired.
    This shows that every $t$-vertex tree is a minor of $B_t$, 
    and concludes the proof of the lemma.
\end{proof}

\begin{lemma}\label{lemma:no_Ct_implies_small_pathwidth}
    Let $t$ be a positive integer,
    let $G$ be a graph, and
    let $S \subseteq V(G)$.
    If $\pw(G,S) \geq 3 \cdot 2^{t+1}-5$, then $\mathcal{C}_t(G,S) \neq \emptyset$. 
\end{lemma}

\begin{proof}
    Suppose $\pw(G,S) \geq 3\cdot 2^{t+1}-5 = 2\cdot \left (3 \cdot 2^{t}-2\right) -1$. 
    Since the complete ternary tree of height $t$ has $3 \cdot 2^{t}-2$ vertices, 
    by \Cref{diestelSrootedoriginal}, there is an $S$-rooted model of the complete ternary tree of height $t$ in $G$.
    A straightforward induction on $t$ then shows that $\mathcal{C}_t(G,S) \neq \emptyset$.
\end{proof}

The following lemma involves a rooted variant of the classical notion of `bramble' from graph minor theory. 

\begin{lemma}\label{lemma:big_bramble_implies_rooted_tree}
    Let $G$ be a graph,
    let $S \subseteq V(G)$,
    and let $\mathcal{B}$ be a set of connected 
    subgraphs of $G$ such that every two subgraphs in $\mathcal{B}$ have at least one vertex in common, and every 
    subgraph in $\mathcal{B}$ contains at least one vertex from $S$. 
    Let $t$ be a positive integer, and suppose that there is no set $X$ of at most $3 \cdot 2^{t+1}-5$ 
    vertices of $G$ such that $X$ meets all members of $\mathcal{B}$. 
    Then, $\mathcal{C}_t(G,S) \neq \emptyset$.
\end{lemma}

\begin{proof}
    By \Cref{lemma:no_Ct_implies_small_pathwidth},
    it is enough to show that $\pw(G,S) \geq 3 \cdot 2^{t+1}-5$.
    Suppose for a contradiction that there is a path decomposition
    $(W_1, \dots, W_\ell)$ of $(G,S)$ of width less than $3 \cdot 2^{t+1}-5$.
    For every $B \in \mathcal{B}$,
    the set $I_B = \{i \in [\ell] \mid W_i \cap V(B) \neq \emptyset\}$ 
    is a nonempty interval, since $B$ is connected and contains some vertex from $S$.  
    Since there are no two vertex-disjoint members of $\mathcal{B}$,
    the intervals $I_B$ for $B \in \mathcal{B}$ pairwise intersect.
    By the Helly property of intervals, there is $i \in [\ell]$
    such that $W_i$ intersects every $B \in \mathcal{B}$.
    Since $|W_i| \leq 3 \cdot 2^{t+1}-5$, this is a contradiction.
\end{proof}

\begin{lemma}\label{lemma:K13_in_big_forests}
    Let $F$ be a forest of maximum degree $d$ and   
    with $c$ connected components. 
    Let $S$ be a nonempty set of leaves of $F$. 
    Then, there is a collection $\mathcal{H}$ of vertex-disjoint subtrees of $F$ such that
    \begin{enumerate}[label={\normalfont(\alph*)}]
        \item $|\mathcal{H}| \geq \frac{|S|-2c}{2d}$, and\label{item:lemma:K13_in_big_forest:H_big}
        \item in every $H \in \mathcal{H}$,
            there is a model of $K_{1,3}$ in which every leaf branch set contains a vertex from $S$.
            \label{item:lemma:K13_in_big_forest:lot_of_K13}
    \end{enumerate}
\end{lemma}

\begin{proof}
    We proceed by induction on $c+|S|+|V(F)|$.
    If $|S| \leq 2c$, then 
    the result follows with $\mathcal{H} = \emptyset$.
    Now suppose $|S| > 2c$.
    
    If $c>1$, then we apply the induction hypothesis on each connected component $C$ of $F$ that contains at least one vertex from $S$, 
    to obtain a collection $\mathcal{H}_C$ of size at least $\frac{|S \cap V(C)| - 2}{2d}$
    satisfying \ref{item:lemma:K13_in_big_forest:lot_of_K13}.
    Then, the union of these families is as desired.
    Now suppose $c=1$, that is, that $F$ is connected.

    If there is a leaf $x$ of $F$ which does not belong to $S$,
    then it is enough to call the induction hypothesis on $F-x$.
    Now suppose that every leaf of $F$ belongs to $S$. 
    Recall that $|S| > 2c$, thus $F$ has at least three leaves, and at least one non-leaf vertex.  
    Choose an arbitrary non-leaf vertex $r$ and fix it as the root of $F$. 
    Let $u$ be a vertex of $F$ at maximum distance from $r$
    such that the subtree $F_u$ of $F$ rooted at $u$ contains at least three leaves.
    Observe that
    \begin{enumerate}
        \item there is a model of $K_{1,3}$ in $F_u$ in which every leaf branch set 
            contains a vertex from $S$, and
        \item $F_u$ contains at most $2d$ leaves of $F$.
    \end{enumerate}    
    Then, by the induction hypothesis applied to every connected component of $F - V(F_u)$ that contains a vertex from $S$,
    there is a collection $\mathcal{H}'$ of vertex-disjoint subtrees of $F - V(F_u)$ such that
    $|\mathcal{H}'| \geq \frac{|S|-2d-2}{2d}$ and for every $H \in \mathcal{H}'$,
    there is a model of $K_{1,3}$ in which every leaf branch set contains a vertex from $S$.
    It then follows that $\mathcal{H} = \mathcal{H}' \cup \{F_u\}$ is as desired. 
\end{proof}

The following lemma will be our main tool to build elements in $\mathcal{C}_t(G,S)$.
Informally, it states that the rooted models of $K_{1,3}$ and $K_3$ have the
Erd\H{o}s-P\'osa property with a linear bounding function.

\begin{lemma}\label{lemma:excluding_K13_and_K3}
    Let $G$ be a graph and let $S \subseteq V(G)$.
    For every positive integer $k$, either
    \begin{enumerate}[label={\normalfont(\arabic*)}]
        \item there are $k$ pairwise vertex-disjoint subgraphs $H_1, \dots, H_k$ of $G$
            such that, for every $i \in [k]$, $H_i$ contains an $S$-rooted model of 
            $K_3$ or a model of $K_{1,3}$ in which all the leaf branch sets intersect $S$; or \label{prop:packing_of_K3_or_K13}
        \item \label{prop:hitting_paths} 
        there are a set $Z \subseteq V(G)$ with $|Z| \leq 11 (k-1)$ and
            an ordering $s_1, \dots, s_\ell$ of $S \setminus Z$
            such that, for every $i\in [\ell]$, $Z \cup \{s_i\}$
            intersects every $\{s_j \mid 1 \leq j < i \}$--$\{s_j \mid i < j \leq \ell\}$ path in $G$.
    \end{enumerate}
\end{lemma}

\begin{proof}
    Let $k$ be a positive integer,
    and let $\mathcal{C}$ be the family of all the connected subgraphs of $G$ containing an $S$-rooted model of $K_3$ or a model of $K_{1,3}$ in which every leaf branch set intersects $S$.
    Suppose \ref{prop:packing_of_K3_or_K13} does not hold: there are no $k$ pairwise vertex-disjoint members of $\mathcal{C}$. 
    We will show that \ref{prop:hitting_paths} holds. 
    
    Let $\mathcal{C}_1$ be the family of all the connected subgraphs $H$ of $G$
    containing  a model of $K_{1,3}$ in which all the leaf branch sets intersect $S$.
    First, we show that there is a set $Z \subseteq V(G)$ of size at most $11(k-1)$
    that intersects every member of $\mathcal{C}_1$.
    Let $\mathcal{H}$ be a collection of pairwise vertex-disjoint members of $\mathcal{C}_1$
    with $|\mathcal{H}|$ maximum.
    By taking the members $H$ of $\mathcal{H}$ with $V(H),E(H)$ inclusion-wise minimal,
    we have that every $H \in \mathcal{H}$ consists in three internally vertex-disjoint
    paths $Q_{H,1}, Q_{H,2}, Q_{H,3}$ sharing exactly one endpoint $v_H$, 
    and with the other endpoints in $S$.
    Moreover, the interiors of $Q_{H,i}$ for $H \in \mathcal{H}$ and $i \in \{1,2,3\}$
    are disjoint from $S$. 

    Let $\mathcal{P}$ be a maximum-size collection of pairwise vertex-disjoint 
    $V(\bigcup \mathcal{H})$--$S$ paths in $G$. 
    (Here, the notation $\bigcup \mathcal{H}$ denotes the graph that is the union of the graphs in $\mathcal{H}$.)
    Note that these paths may have length $0$.
 
    \begin{claim}\label{claim:excludingK13}
        For every $H \in \mathcal{H}$ and $i \in \{1,2,3\}$,
        at most two paths in $\mathcal{P}$ have 
        an endpoint in the interior of $Q_{H,i}$.
    \end{claim}

    \begin{proofclaim}
        Suppose there are $H \in \mathcal{H}$, $i \in \{1,2,3\}$
        and $P_1, P_2, P_3 \in \mathcal{P}$ with $P_1,P_2,P_3$ distinct and each having an endpoint
        in the interior of $Q_{H,i}$.
        Note that these paths have length at least one, since the interior of $Q_{H,i}$
        is disjoint from $S$.
        Then, $H \cup P_1 \cup P_2 \cup P_3$ contains two vertex-disjoint subgraphs $H_1, H_2$
        both containing a model of $K_{1,3}$
        in which every leaf branch set intersects $S$ (see \Cref{fig:claimK13}).
        But then, the family $(\mathcal{H} \setminus \{H\}) \cup \{H_1, H_2\}$
        contradicts the maximality of $|\mathcal{H}|$.
    \end{proofclaim}

    \begin{figure}
        \centering
        \includegraphics{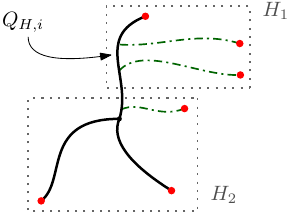}
        \caption{Proof of \Cref{claim:excludingK13}. 
            The marked vertices, in red, are vertices in $S$, 
            and the dashed paths, in green, are the paths $P_1,P_2,P_3 \in \mathcal{P}$.
            We find two vertex-disjoint subgraphs $H_1,H_2$ both containing 
            a model of $K_{1,3}$ in which leaf branch sets intersect $S$.}
        \label{fig:claimK13}
    \end{figure}

    Since every $H \in \mathcal{H}$ has $4$ vertices not in the interior of one of 
    $Q_{H,1},Q_{H,2},Q_{H,3}$, 
    it follows from \Cref{claim:excludingK13} that 
    \[
        |\mathcal{P}| \leq (4 + 3 \cdot 2) |\mathcal{H}| \leq 10(k-1).
    \]
    Hence, by Menger's Theorem, there is a set $Z_0 \subseteq V(G)$
    of size at most $10(k-1)$ that intersects every
    $V(\bigcup \mathcal{H})$--$S$ path in $G$.
    Now, for every connected subgraph $H\in \mathcal{C}_1$,
    we have that $H$ intersects $S$,
    as well as $V(\bigcup \mathcal{H})$ by maximality of $|\mathcal{H}|$,
    and so it follows that $Z_0$ meets $V(H)$.
    Therefore, 
    \begin{equation}\label{eq:Z0_hits_K13}
        \text{$G-Z_0$ has no model of $K_{1,3}$ in which 
            every leaf branch set intersects $S$.}\tag{$\star$}
    \end{equation}

    Let $C$ be a connected component of $G-Z_0$.
    If $C$ has an $S$-rooted model of $K_3$,
    then choose a vertex $s_C \in S \cap V(C)$.
    If there is still an $S$-rooted model of $K_3$ in $C-s_C$,
    then by considering a path from $s_C$ to this model in $C$,
    we deduce that $C$ contains an $S$-rooted model of $K_{1,3}$,
    contradicting \eqref{eq:Z0_hits_K13}.
    Therefore, $C-s_C$ has no $S$-rooted model of $K_3$.

    Let 
    \[
        Z = Z_0\cup  \{s_C \mid \text{$C$ connected component of $G-Z_0$
            having an $S$-rooted model of $K_3$}\}.
    \]
    Since there are no $k$ pairwise disjoint $S$-rooted models of $K_3$ in $G$,
    we have 
    \[
        |Z| \leq |Z_0| + (k-1) \leq 11(k-1).
    \]
    Moreover, by construction, $G-Z$ has no $S$-rooted model of $K_3$, nor
    any model of $K_{1,3}$ with all leaf branch sets intersecting $S$.

    This proves that $Z$ intersects every member of $\mathcal{C}$.
    In particular, every connected component $C$ of $G-(Z \cup S)$
    has at most $2$ neighbors in $S \setminus Z$.
    Now consider the graph $G'$ with vertex set $S \setminus Z$
    and edges all the pairs $ss'$ such that there is an $s$--$s'$ path in $G-Z$ internally disjoint from $S\setminus Z$ (this is the so-called ``torso'' of $S\setminus Z$ in $G-Z$).
    Since $Z$ intersects every member of $\mathcal{C}$,
    $G'$ has maximum degree at most $2$ and has no cycle.
    Hence $G'$ is a disjoint union of paths, and the result follows.
\end{proof}

\section{Proof of main result}
\label{sec:proof}

The heart of the proof of \cref{thm:main} is the following technical theorem. 

\begin{theorem}
    \label{thm:technical}
    There exists a function $h\colon \N \to \N$ such that, 
    for every graph $G$, every $S\subseteq V(G)$, and every integers $t,k$ with $t\geq 0$ and $k\geq 1$, at least one the following two properties holds: 
    \begin{enumerate}[label={\normalfont(\arabic*)}]
        \item $\C_t(G, S)$ contains $k$ vertex-disjoint subgraphs of $G$, 
        \item \label{outcome_X} there is $X\subseteq V(G)$ with $|X| \leq h(t) k$ such that $X$ meets all subgraphs of $G$ in $\C_t(G, S)$.
    \end{enumerate}
\end{theorem}

Before proving \cref{thm:technical}, let us prove first that it quickly implies \cref{thm:main}. 

\begin{proof}[Proof of \cref{thm:main} assuming \cref{thm:technical}]
    We define $g \colon \N \to \N$ by $g(t) = h(t) + (3^{t+1}-2)$
    for every $t \in \N$, where $h(\cdot)$ is the function from \Cref{thm:technical}.
    Let $k$ be a positive integer,
    let $T$ be a tree on $t$ vertices,
    let $G$ be a graph, and
    let $S \subseteq V(G)$.
    Suppose that there are no $k$ pairwise vertex-disjoint subgraphs of $G$,
    each containing an $S$-rooted model of $T$.
    Then, by \Cref{lemma:properties_of_Ct}.\ref{prop:every Ct_contains_every_tree},
    every $H \in \mathcal{C}_t(G,S)$ contains an $S$-rooted model of $T$, and so
    there are no $k$ pairwise vertex-disjoint members of $\mathcal{C}_t(G,S)$.
    By \Cref{thm:technical}, there is a set $X_0 \subseteq V(G)$ 
    with $|X_0| \leq h(t) k$ such that $X_0$ meets all subgraphs of $G$ in $\mathcal{C}_t(G,S)$.
    In other words, $\mathcal{C}_t(G-X_0,S \setminus X_0) = \emptyset$.
    By \Cref{lemma:no_Ct_implies_small_pathwidth}, this implies that
    $\pw(G-X_0, S \setminus X_0) \leq 3 \cdot 2^{t+1}-6 \leq 3^{t+1}-3$. 
    Then, by \Cref{lem:erdos_posa_pw},
    there is a set $X_1 \subseteq V(G-X_0)$ of size at most $(3^{t+1}-2)(k-1)$
    that intersects every $S$-rooted model of $T$ in $G-X_0$. 
    Then, for  $X = X_0 \cup X_1$,
    we have that $G-X$ has no $S$-rooted model of $T$, 
    and $|X| \leq h(t)k+(3^{t+1}-2)k = g(t) \cdot k$.
\end{proof}

We may now turn to the proof of \cref{thm:technical}. 

\begin{proof}[Proof of \cref{thm:technical}]
    Let $h\colon \NN \to \NN$ be defined by, for every $t \in \N$,
    \[
        h(t) = 
        \begin{cases}
            1 & \textrm{if $t=0$} \\
            3^{t+10}t \cdot h(t-1) & \textrm{if $t \geq 1$.}
        \end{cases}
    \]
    
    We proceed by induction on $t$.
    When $t=0$, the family $\mathcal{C}_0(G,S)$ consists of
    the $1$-vertex subgraphs of $G$ intersecting $S$. 
    Hence, for every $k \geq 1$, either $|S| \geq k$
    and so there are $k$ pairwise vertex-disjoint members of $\mathcal{C}_0(G,S)$,
    or $|S| \leq k-1$, and so $X=S$ is as wanted.

    Now suppose $t \geq 1$ and that the result holds for $t-1$.
    Let $k$ be a positive integer,
    let $G$ be a graph,
    and let $S \subseteq V(G)$ be such that there are no $k$ pairwise vertex-disjoint
    members of $\mathcal{C}_t(G,S)$.
    We have to show that \ref{outcome_X} holds. 
    
    Let $\mathcal{P} \subseteq \mathcal{C}_{t-1}(G,S)$ 
    be such that
    the members of $\mathcal{P}$ are pairwise vertex-disjoint and $|\mathcal{P}|$ is maximum.
    Note that the maximality of $|\mathcal{P}|$ implies the following.

    \begin{claim}\label{claim:maximality_of_P}
        For every subgraph $H$ of $G$,
        there are no $|\{P \in \mathcal{P} \mid V(P) \cap V(H) \neq \emptyset\}|+1$
        pairwise vertex-disjoint members of $\mathcal{C}_{t-1}(H,S)$.
    \end{claim}

    \begin{proofclaim}
        If there is a family $\mathcal{P}'$ of
        $|\{P \in \mathcal{P} \mid V(P) \cap V(H) \neq \emptyset\}|+1$
        vertex-disjoint members of $\mathcal{C}_{t-1}(H,S)$,
        then $(\mathcal{P} \setminus \{P \in \mathcal{P} \mid V(P) \cap V(H) \neq \emptyset\}) \cup \mathcal{P}'$ is a family of $|\mathcal{P}|+1$
        pairwise vertex-disjoint members of $\mathcal{C}_{t-1}(G,S)$,
        contradicting the maximality of $|\mathcal{P}|$.
    \end{proofclaim}

    In particular, if $|\mathcal{P}| < (3^{t+1}+11)k$,
    then, 
    there are no $(3^{t+1}+11)k$
    pairwise vertex-disjoint members of $\mathcal{C}_{t-1}(G,S)$. 
    It follows by the induction hypothesis that there is a set $X$ of at most
    $(3^{t+1}+11)h(t-1)k \leq h(t)k$ vertices of $G$ meeting every member of $\mathcal{C}_{t-1}(G,S)$, and so every member of $\mathcal{C}_t(G,S)$ as well.
    Now assume $|\mathcal{P}| \geq (3^{t+1}+11)k$.

    Let $G'$ be the graph obtained from $G$ by contracting every $P \in \mathcal{P}$
    into a single vertex $u_P$,
    and let $S' = \{u_P \mid P \in \mathcal{P}\}$.

    \begin{claim}
        There are no $k$ pairwise vertex-disjoint subgraphs $H_1, \dots, H_k$ of $G'$
        such that,  for every $i \in [k]$, 
        $H_i$ contains either an $S'$-rooted model of $K_3$ or
        a model of $K_{1,3}$ in which every leaf branch set intersects $S'$.
    \end{claim}

    \begin{proofclaim}
        Otherwise, for every $i \in [k]$,
        $G[(V(H_i) \setminus S') \cup \bigcup_{P \in \mathcal{P}, u_P \in V(H_i)} V(P)\}]$
        contains a member of $\mathcal{C}_t(G,S)$ by construction,
        and so there are $k$ pairwise vertex-disjoint members of $\mathcal{C}_t(G,S)$ in $G$,
        a contradiction.
    \end{proofclaim}

    Hence, by \Cref{lemma:excluding_K13_and_K3} applied to $(G',S')$, 
    there is a set $Z_0 \subseteq V(G)\setminus (\bigcup_{P\in \mathcal P}V(P))$ 
    and $Z'_0 \subseteq \mathcal{P}$ with 
    \[
        |Z_0| + |Z'_0| \leq 11 (k-1),
    \]
    and an ordering $P_1, \dots, P_\ell$ of $\mathcal{P} \setminus Z'_0$
    such that:
    \begin{equation}
        \text{For every $i \in [\ell]$,
              $Z_0 \cup V(\textstyle\bigcup Z'_0) \cup P_i$ meets all 
              $\textstyle\bigcup_{1 \leq j < i} V(P_j)$--$\textstyle\bigcup_{i<j\leq\ell} V(P_j)$ paths in $G$.} 
              \tag{$\star$}\label{eq:path_partition}
    \end{equation} 
    We fix such sets $Z_0,Z'_0$ and ordering $P_1, \dots, P_\ell$.
    Let $G''$ be the graph obtained from $G$ 
    by contracting every $P \in \mathcal{P} \setminus Z'_0$ into a single vertex $u_P$,
    and let $S'' = \{u_P \mid P \in \mathcal{P} \setminus Z'_0\}$.

    \begin{claim}
        There are no $(4t+33)k$ pairwise vertex-disjoint $S''$--$(Z_0 \cup V(\bigcup Z'_0))$ 
        paths in $G''$.
    \end{claim}

    \begin{proofclaim}
        Suppose for contradiction that there is a family $\mathcal{Q}$ of $(4t+33)k$
        pairwise vertex-disjoint $S''$--$(Z_0 \cup V(\bigcup Z'_0))$ paths in $G''$.

        Let $\mathcal{Q}_0$ be the set of all the paths in $\mathcal{Q}$
        with one endpoint in $Z_0$,
        and let $\mathcal{Q}'_0 = \mathcal{Q} \setminus \mathcal{Q}_0$.
        Note that $|\mathcal{Q}_0| \leq |Z_0| \leq 11 k$,
        and so $|\mathcal{Q}'_0| \geq (4t+33)k - 11k = (4t+22) k$.
        
        For every $P \in Z'_0$,
        we have $P \in \mathcal{C}_{t-1}(G,S)$,
        and so by \Cref{lemma:properties_of_Ct}.\ref{prop:max_degree},
        $P$ is connected and has maximum degree at most $2t-1$.
        Thus, we can fix a spanning tree $T_P$ of $P$ of maximum degree at most $2t-1$
        for each $P \in Z'_0$.
        Let $F$ be the union of these spanning trees $T_P$ for $P \in Z'_0$ 
        together with the paths in $\mathcal{Q}'_0$,
        and let $R$ be the set of the endpoints in $S''$ of the paths in $\mathcal{Q}'_0$.
        Then $F$ is a forest of maximum degree at most $2t$
        and $R$ is a subset of leaves of $F$. 
        Furthermore, $R\neq \emptyset$ since $\mathcal{Q}'_0$ is non empty. 
        By \Cref{lemma:K13_in_big_forests},
        there are at least 
        \[\frac{|R|-2|Z'_0|}{4t} = \frac{|\mathcal{Q}'_0|-2|Z'_0|}{4t}
        \geq \frac{(4t + 22)k - 2 \cdot 11 \cdot k}{4t} = k\] 
        pairwise vertex-disjoint subgraphs of $F$ such that
        each of them contains a model of $K_{1,3}$
        in which every leaf branch set intersects $S''$.
        Each such subgraph $H''$ of $G''$ yields a subgraph $G[(V(H'')\setminus S'') \cup \bigcup_{P \in \mathcal{P}, u_P \in V(H'')} V(P)]$ of $G$ containing
        a member of $\mathcal{C}_t(G,S)$,
        and so there are $k$ pairwise vertex-disjoint members of $\mathcal{C}_t(G,S)$,
        a contradiction.
    \end{proofclaim}
    By Menger's Theorem applied to $G''$,
    we deduce that there exist $Z_1 \subseteq V(G)\setminus  (\bigcup_{P\in \mathcal P \setminus Z'_0}V(P))$
    and
    $Z'_1 \subseteq \mathcal{P} \setminus Z'_0$
    such that 
    \[
        |Z_1| + |Z'_1| \leq (4t+33)k
    \]
    and $Z_1 \cup V(\bigcup Z'_1)$ intersects every 
    $S''$--$(Z_0 \cup V(\bigcup Z'_0))$ path in $G$.

    For every $H \in \mathcal{C}_t(G-(Z_0 \cup V(\bigcup Z'_0)), S)$,
    let $I_H = \{i \in [\ell] \mid V(H) \cap V(P_i) \neq \emptyset\}$.
    By the maximality of $|\mathcal{P}|$,
    $H$ intersects $\mathcal{P} \setminus Z'_0$, and so $I_H$ is nonempty.
    Moreover, because $H$ is connected and by \eqref{eq:path_partition},
    we have that $I_H$ is an interval.
    By the Helly property of intervals,
    and because there are no $k$ pairwise vertex-disjoint members of $\mathcal{C}_t(G,S)$,
    there exists $J \subseteq [\ell]$ of size at most $k-1$
    that meets every $I_H$ for $H \in \mathcal{C}_t(G-(Z_0\cup \bigcup Z'_0),S)$.
    Let 
    \[
        Z'_2 = \{P_j \mid j \in J\}.
    \]
    Note that $V(\bigcup Z'_2)$ meets every member of $\mathcal{C}_t(G - (Z_0 \cup V(\bigcup Z'_0)), S)$.
    See \Cref{fig:main_proof} for an illustration of the objects considered so far.

    \begin{figure}
        \centering
        \includegraphics{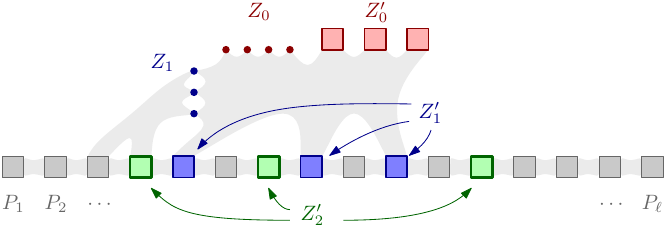}
        \caption{Illustration for the proof of \Cref{thm:technical}. The square shapes represent the elements of $\mathcal{P}$ and the potential connections between the considered objects are depicted in light gray.}
        \label{fig:main_proof}
    \end{figure}

    We partition the set $[\ell]$ into intervals $[i(1),i(2)-1], \dots, [i(m),i(m+1)-1]$ of size at least $3 \cdot 2^{t+1}-5$ and at most
    $2(3^{t+1}-2)-1$. This is possible since $\ell = |\mathcal{P}| - |Z'_0| \geq(3^{t+1}+11)k - 11k \geq 3\cdot 2^{t+1}-5$. 
    For every $j \in [m]$,
    we say that $[i(j), i(j+1)-1]$ is \defin{special} if there are $3\cdot 2^{t+1}-5$
    pairwise vertex-disjoint $V(P_{i(j)})$--$V(P_{i(j+1)-1})$ paths in $G - (Z_0 \cup V(\bigcup Z'_0))$.

    \begin{claim}
        There are at most $k-1$ special intervals among $[i(1),i(2)-1], \dots, [i(m), i(m+1)-1]$.
    \end{claim}

    \begin{proofclaim}
        Let $j \in [m]$,
        and suppose that the interval $[i(j), i(j+1)-1]$ is special.
        Let $G_j$ be the subgraph of $G$ induced by the union of the vertex set of
        all the $P_a$ for $a \in [i(j), i(j+1)-1]$ with the vertex sets of all
        the connected components $C$ of $G-(Z_0 \cup V(\bigcup \mathcal{P}))$ that have a neighbor
        in $\bigcup_{a \in [i(j)+1, i(j+1)-2]} V(P_a)$. 
        Since $[i(j), i(j+1)-1]$ is special, there is a family $\mathcal{Q}_j$ of $3\cdot 2^{t+1}-5$ pairwise vertex-disjoint $V(P_{i(j)})$--$V(P_{i(j+1)-1})$ paths in  $G - (Z_0 \cup V(\bigcup Z'_0))$, which are thus in $G_j$. 
        Now, let $\mathcal{H} = \{G[V(P_a) \cup V(Q)] \mid Q \in \mathcal{Q}_j, a \in [i(j), i(j+1)-1]\}$.
        Observe that, since every $P \in \mathcal{P}$ is connected by \Cref{lemma:properties_of_Ct}.\ref{prop:max_degree},
        every member of $\mathcal{H}$ is a connected subgraph of $G_j$
        that intersects $S$.
        Moreover, these subgraphs pairwise intersect,
        and there is no set of less than $3\cdot 2^{t+1}-5$ vertices meeting all of them.
        Hence, by \Cref{lemma:big_bramble_implies_rooted_tree},
        we have $\mathcal{C}_t(G_j,S) \neq \emptyset$,
        and so $G_j$ contains a member of $\mathcal{C}_t(G,S)$.

        Since the graphs $G_j$ for $j \in [m]$ are pairwise vertex-disjoint,
        and because there are no $k$ pairwise vertex-disjoint members of $\mathcal{C}_t(G,S)$,
        we deduce that there are at most $k-1$ special intervals among $[i(1), i(2)-1], \dots [i(m), i(m+1)-1]$.
    \end{proofclaim}

    We say that an interval $[i(j), i(j+1)-1]$ is \defin{marked}
    if
    \begin{enumerate}[label=(M\arabic*)]
        \item $[i(j), i(j+1)-1]$ is special, or
        \item $Z'_1 \cap \{P_a \mid a \in [i(j),i(j+1)-1]\} \neq \emptyset$, or
        \item $Z'_2 \cap \{P_a \mid a \in [i(j),i(j+1)-1]\} \neq \emptyset$.
    \end{enumerate}

    Note that the number of marked intervals is at most 
    \[
        (k-1) + |Z'_1| + |Z'_2| \leq (4t+35) k.
    \]

    An interval $[i(j),i(j+1)-1]$ is \defin{bordering}
    if it is not a marked interval but either $[i(j-1),i(j)+1]$ or $[i(j+1),i(j+2)-1]$ is.
    Note that there are at most $2(4t+35)k$ bordering intervals.
    Let $B \subseteq [m]$ be set of all the indices $j$ such that $[i(j), i(j+1)-1]$ is bordering.
    For each $j \in B$, 
    $[i(j),i(j+1)-1]$ is not special, and so there are no $3^{t+2}-2$ pairwise vertex-disjoint
    $V(P_{i(j)})$--$V(P_{i(j+1)-1})$ paths in $G - (Z_0 \cup V(\bigcup Z'_0))$.
    By Menger's Theorem, there is a set $X_j$ of at most $3\cdot 2^{t+2}-6$ vertices that intersects every such path.

    Let
    \[
        Z_3 = \textstyle\bigcup_{j \in B} X_j,
    \]
    and let
    \[
        Z' = Z'_0 \cup \big\{P_a \mid a \in [i(j),i(j+1)-1], \text{$[i(j),i(j+1)-1]$ marked or bordering interval}\big\}.
    \]
    Note that $Z'_0 \cup Z'_1 \cup Z'_2 \subseteq Z'$, and
    \[
        |Z'| \leq 11k + (2(3\cdot 2^{t+1}-5)-1)\cdot 3(4t+35)k \leq 3^{t+8}t \cdot k.
    \]

    \begin{claim}\label{claim:Z'_separates}
        Every $\left(\bigcup_{P \in Z'_0 \cup Z'_1 \cup Z'_2} V(P)\right)$--$\left(\bigcup_{P \in \mathcal{P} \setminus Z'} V(P)\right)$
        path in $G$ intersects $Z_0 \cup Z_1 \cup Z_3$.
    \end{claim}

    \begin{proofclaim}
        Let $P \in \mathcal{P} \setminus Z'$
        and $P' \in Z'_0 \cup Z'_1 \cup Z'_2$,
        and let $Q$ be a $V(P)$--$V(P')$ path in $G$ vertex-disjoint from $Z_0 \cup Z_1$.
        Then, there is an interval $[i(j),i(j+1)-1]$ 
        which is neither special nor bordering
        and such that $P \in \{P_a \mid a \in [i(j),i(j+1)-1]\}$.
        If $P' \in Z'_0 \cup Z'_1$,
        then by the definition of $Z_1$ and $Z'_1$, 
        we have that $Q$ intersects $\bigcup_{P'' \in Z'_1} V(P'')$. 
        Otherwise, we still have that $Q$ intersects $\bigcup_{P'' \in Z'_2} V(P'')$.
        In both cases,
        we have that $Q$ intersects $V(P'')$ for some $P'' \in Z'_1 \cup Z'_2$.
        Fix such a $P''$ such that the subpath $Q'$ of $Q$ between $P$ and $P''$ has
        minimum length.
        In particular, $V(Q')$ is disjoint from $V(\bigcup Z'_0)$.
        Let $j'' \in [m]$ be such that $P'' \in \{P_a \mid a \in [i(j''),i(j''+1)-1]\}$.
        By definition, the interval $[i(j''),i(j''+1)-1]$ is marked. 
        Hence, there is a bordering interval $[i(j'),i(j'+1)-1]$
        with $j'' < j' <j$ or $j < j' < j''$.
        In both cases,
        since $Q'$ is a path in $G - (Z_0 \cup V(\bigcup Z'_0))$,
        the path $Q'$ intersects $P_{i(j')}$ and $P_{i(j'+1)-1}$ 
        by \eqref{eq:path_partition},
        and so $X_{j'} \subseteq Z_3$ intersects $Q$ by the definition of $X_{j'}$.
    \end{proofclaim}
    
    Let $G_0$ be the subgraph of $G$ induced by the union of the vertex sets
    of all the connected components $C$ of $G - (Z_0 \cup Z_1 \cup Z_3)$ that intersect $V\big(\bigcup Z'_0 \cup \bigcup Z'_1 \cup \bigcup Z'_2\big)$. 
    By \Cref{claim:Z'_separates}, $\{P \in \mathcal{P} \mid V(P) \cap V(G_0) \neq \emptyset\} \subseteq Z'$. 
    Therefore, by \Cref{claim:maximality_of_P},
    there are no $|Z'|+1$ pairwise vertex-disjoint members of $\mathcal{C}_{t-1}(G_0,S)$.
    By the induction hypothesis,
    it follows that
    there is a set $Z_4 \subseteq V(G_0)$ of size at most $h(t-1)(|Z'|+1) \leq h(t-1) \cdot 3|Z'| \leq 3^{t+9}t \cdot h(t-1)\cdot k$ 
    that meets every member of $\mathcal{C}_{t-1}(G_0,S)$,
    and so every member of $\mathcal{C}_t(G_0,S)$.

    Finally, we set
    \[
        X = Z_0 \cup Z_1 \cup Z_3 \cup Z_4.
    \]
    Note that  
    \begin{align*}
        |X|
        &\leq |Z_0| + |Z_1| + |Z_3| + |Z_4| \\
        &\leq 11 \cdot k + (4t+33) \cdot k + 2(4t+33)(3\cdot 2^{t+1}-6) \cdot k + 3^{t+9}t \cdot h(t-1) \cdot k \\
        &\leq \left(11 + 2(4t+33)3^{t+2} + 3^{t+9}t \cdot h(t-1)\right)\cdot k \\
        &\leq \left(11 + 3^{t+8}t + 3^{t+9}t \cdot h(t-1)\right)\cdot k \\
        &\leq 3^{t+10}t \cdot h(t-1) \cdot k = h(t) \cdot k,
    \end{align*}
    and so it only remains to show that
    $X$ meets every member of $\mathcal{C}_t(G,S)$.
    
    Let $H \in \mathcal{C}_t(G,S)$ and suppose for contradiction that
    $V(H)$ is disjoint from $X$.
    Since $H$ contains a member of $\mathcal{C}_{t-1}(G,S)$,
    and by the maximality of $\mathcal{P}$, 
    we have that $V(H)$ intersects at least one of $V(P)$ for $P \in \mathcal{P}$.
    Moreover, by the definition of $Z'_2$, and because $V(H)$ is disjoint from $Z_0$, $V(H)$ intersects $V(P)$ for some $P \in Z'_0 \cup Z'_2$.
    Hence, since $V(H)$ is disjoint from $Z_3$,
    $H$ is a subgraph of $G_0$, and so is a member of $\mathcal{C}_t(G_0,S)$.
    Therefore, $V(H)$ intersects $Z_4 \subseteq X$, a contradiction.
    This shows that $X$ meets every member of $\mathcal{C}_t(G,S)$,
    and concludes the induction.
\end{proof}

\section{Generalization to arbitrarily rooted forests}
\label{sec:forests}

In this section, we prove \Cref{thm:main_arbitrarily_rooted_forest}, about models of $(F,R)$ in $(G, S)$ for a forest $F$ and $R\subseteq V(F)$, and a  graph $G$ and $S\subseteq V(G)$.
The proof consists in first applying \Cref{thm:main} to a tree $T$ containing
$F$ as a spanning subgraph to hit all the $S$-rooted models of $T$, 
and then using both standard techniques on graphs of bounded pathwidth
and ideas from \citet{DJMM25} to finally hit every model of $(F,R)$ in $(G, S)$. 
In particular, we will rely on the following lemma,
which appears implicitly in \citet{D95}.

\begin{lemma}[\citet{D95}]\label{lemma:Diestel_lemma}
    Let $G$ be a graph and let $T$ be a tree.
    If $\pw(G) \geq |V(T)|-1$, then
    there is a separation $(A,B)$ of $G$ such that
    \begin{enumerate}[label={\normalfont(\alph*)}]
        \item there is a path decomposition $(W_1, \dots, W_\ell)$ of $G[A]$
            of width at most $|V(T)|-1$
            such that $A \cap B \subseteq W_\ell$, and
        \item there is an $(A\cap B)$-rooted model of $T$ in $G[A]$.
    \end{enumerate}
\end{lemma}

We will use \Cref{lemma:Diestel_lemma} through the following corollary
which is implicitly used in \citet{DJMM25}.

\begin{lemma}\label{lemma:Diestel_lemma_consequence}
    Let $G$ be a graph and let $T_1, \dots, T_c$ be trees.
    If at least one of $T_1, \dots, T_c$ is a minor of $G$, then there are $i_0 \in [c]$ 
    and a separation $(A,B)$ of $G$ of order at most $|V(T_{i_0})|$ such that
    \begin{enumerate}[label={\normalfont(\alph*)}]
        \item $T_{i_0}$ is a minor of $G[A]$, and
        \item for every $i \in [c]$, $T_i$ is not a minor of $G[A \setminus B]$.
    \end{enumerate}
\end{lemma}

\begin{proof}
    Let $i_1 \in [c]$ be such that $T_{i_1}$ is a minor of $G$ and 
    $|V(T_{i_1})|$ is minimum with this property.
    If $\pw(G) \leq |V(T_{i_1})|-2$, then let
    $(W_1, \dots, W_\ell)$ be a path decomposition of $G$ of minimum width.
    Then, we take $(A,B) = (W_1 \cup \dots \cup W_j, W_j \cup \dots \cup W_{\ell})$
    where $j \in [\ell]$ is minimum such that $G[W_1 \cup \dots \cup W_j]$
    contains one of $T_1, \dots, T_c$ as a minor.

    If $\pw(G) \geq |V(T_{i_1})|-1$,
    then let $(A_0,B_0)$ and $(W_1, \dots, W_\ell)$ be respectively 
    the separation and the path decomposition given by \Cref{lemma:Diestel_lemma}.
    Again, we take $(A,B) = (W_1 \cup \dots \cup W_j, W_j \cup \dots \cup W_{\ell})$ 
    where $j \in [\ell]$ is minimum such that $G[W_1 \cup \dots \cup W_j]$ 
    contains one of $T_1, \dots, T_c$ as a minor.

    In both cases, $G[A]$ contains at least one of $T_1, \dots, T_c$ as a minor,
    say $T_{i_0}$,
    and $(A,B)$ has order at most $|V(T_{i_1})| \leq |V(T_{i_0})|$.
    Moreover, the minimality of $j$ ensures in both cases that, for every $i \in [c]$,
    $T_i$ is not a minor of $G[A \setminus B]$.
\end{proof}

\begin{lemma}\label{lemma:bounded_tw_implies_EP}
    Let $G$ be a graph and let $S \subseteq V(G)$ be such that $\tw(G,S) < w$.
    Let $T_1, \dots, T_c$ be trees,
    let $R_1 \subseteq V(T_1), \dots, R_c \subseteq V(T_c)$,
    and let $x_1, \dots ,x_c$ be nonnegative integers. 
    If there is no family $(G_{i,y})_{i \in [c], y \in [x_i]}$ of pairwise vertex-disjoint
    subgraphs of $G$ such that $(G_{i,y},S)$ contains a model of $(T_i,R_i)$
    for every $i \in [c]$ and $y \in [x_i]$, then
    there are $j \in [c]$ and a set $X \subseteq V(G)$ with 
    \[
        |X| \leq \sum_{i \in [c]} (|V(T_i)| + w) x_i
    \]
    such that $x_j\geq 1$ and there is no model of $(T_j,R_j)$ in $(G-X,S)$.
\end{lemma}

\begin{proof}
    We proceed by induction on $c + \sum_{i \in [c]} x_i$.
    If $c=0$, then the result is vacuously true.
    Now assume $c \geq 1$.
    If there is $i \in [c]$ with $x_i=0$, then
    we remove $(T_i,R_i)$ from the list $(T_1,R_1), \dots, (T_c,R_c)$, and call the induction hypothesis.
    Now assume $x_i \geq 1$ for every $i \in [c]$.

    If there is $i \in [c]$ such that $(T_i,R_i)$ is not a minor of $(G,S)$,
    then the result holds for $Z = \emptyset$.
    Now suppose that all of $(T_1, R_1) ,\dots, (T_c, R_c)$ are minors of $(G,S)$.

    \begin{claim}\label{claim:bd_tw}
        There are $i_0 \in [c]$ and a separation $(A,B)$ of $G$ such that
        \begin{enumerate}[label={\normalfont(\alph*)}]
            \item $|A \cap B| \leq |V(T_{i_0})| + w$, \label{item:claim:bd_tw:order}
            \item there is a model of $(T_{i_0}, R_{i_0})$ in $(G[A], S)$, and  \label{item:claim:bd_tw:there_is_a_model}
            \item for every $i \in [c]$, there is no model of $(T_i, R_i)$ in $(G[A \setminus B], S)$.  \label{item:claim:bd_tw:there_is_no_model}
        \end{enumerate}
    \end{claim}

    \begin{proofclaim}
        Fix a tree decomposition $\big(F, (W_x \mid x \in V(F))\big)$ of $(G,S)$ of width less than $w$.
        We root $F$ at an arbitrary vertex $s$.
        For every $y \in V(F)$, let $F_y$ be the subtree of $F$ rooted at $y$,
        and let $G_y$ be the subgraph of $G$ induced by
        the union of $\bigcup_{x \in V(F_y)} W_x$ with the vertex set of all the connected components $C$
        of $G-\bigcup_{x \in V(F)} W_x$ with $N_G(V(C)) \subseteq \bigcup_{x \in V(F_y)} W_x$.
    
        Let $z$ be a vertex at maximum distance from $s$ in $F$ such that $(G_z,S)$ contains 
        a model of $(T_i,R_i)$ for some $i \in [c]$,
        and let $I = \{i \in [c] \mid R_i = \emptyset \text{ and } T_i \text{ minor of } G_z\}$.
        
        First suppose that $I \neq \emptyset$.
        Then, by \Cref{lemma:Diestel_lemma_consequence} applied to $G_z$ 
        and the family $(T_i)_{i \in I}$, 
        there are $i_0 \in I$ and
        a separation $(A_0,B_0)$ of $G_z$ of order at most $|V(T_{i_0})|$
        such that
        \begin{enumerate}[label={\ref{lemma:Diestel_lemma_consequence}.(\alph*)}]
            \item $T_{i_0}$ is a minor of $G_z[A_0]$, and \label{item:proof_of_claim_bd_tw:Ti0_minor}
            \item for every $i \in I$, $T_i$ is not a minor of $G_z[A_0 \setminus B_0]$. \label{item:proof_of_claim_bd_tw:Ti_not_minor}
        \end{enumerate}
        Let $(A,B) = (A_0, B_0 \cup (V(G) \setminus V(G_z)) \cup W_z)$.
        Note that $(A,B)$ is a separation of $G$,
        and $|A \cap B| \leq |A_0 \cap B_0| + |W_z| \leq |V(T_{i_0})|+w$,
        which proves \ref{item:claim:bd_tw:order}.
        By construction, $G[A] = G_z[A_0]$, and so there is a model of $(T_{i_0}, \emptyset)$ in $(G[A], S)$ by \ref{item:proof_of_claim_bd_tw:Ti0_minor}, which proves \ref{item:claim:bd_tw:there_is_a_model}.
        Moreover,
        for every $i \in [c] \setminus I$,
        since every model of $(T_i,R_i)$ in $(G,S)$ intersects $\bigcup_{x \in V(F)} W_x$, 
        by the definition of $z$, $(T_i,R_i)$ is not a minor of $(G_z - W_z, S)$, 
        and so $(T_i,R_i)$ is not a minor of $(G[A \setminus B], S)$.
        Combined with \ref{item:proof_of_claim_bd_tw:Ti_not_minor}, 
        this proves \ref{item:claim:bd_tw:there_is_no_model}.
        
        Now suppose $I = \emptyset$.
        Let $(A,B) = (V(G_z), (V(G) \setminus V(G_z)) \cup W_z)$,
        and let $i_0 \in [c]$ be such that $(T_{i_0}, R_{i_0})$ is a minor of $(G[A], S \cap A)$.
        Clearly, $(A,B)$ is a separation of $G$ of order at most $w$, and so \ref{item:claim:bd_tw:order} holds.
        Moreover, there is a model of $(T_{i_0}, R_{i_0})$ in $(G[A], S)$, and so \ref{item:claim:bd_tw:there_is_a_model} holds.
        Finally, by the definition of $z$, there is no model of $(T_i,R_i)$ in $(G[A \setminus B], S)$, and so \ref{item:claim:bd_tw:there_is_no_model} holds.
        In both cases, we found a separation $(A,B)$ of $G$ and $i_0 \in [c]$ as wanted.
    \end{proofclaim}

    Fix $i_0 \in [c]$ and $(A,B)$ as in \Cref{claim:bd_tw}.
    Let $G' = G[B \setminus A]$,
    and for every $i \in [c]$, let
    \[
        x'_i = 
        \begin{cases}
            x_{i_0}-1 & \textrm{if $i=i_0$,} \\
            x_i & \textrm{otherwise.}
        \end{cases}
    \]
    Observe that $G'$ does not admit a family
    $(G'_{i,y})_{i \in [c], y \in [x'_i]}$ of pairwise vertex-disjoint subgraphs of $G'$ such that $(G'_{i,y},S)$ contains a model of $(T_i,R_i)$
    for every $i \in [c]$ and $y \in [x'_i]$, as otherwise,
    the family $(G_{i,y})_{i \in [c], y \in [x_i]}$ defined by
    \[
        G_{i,y} =
        \begin{cases}
            G[A] & \textrm{if $(i,y)=(i_0,x_{i_0})$,} \\
            G'_{i,y} & \textrm{otherwise,}
        \end{cases}
    \]
    for $i \in [c]$ and $y \in [x_i]$,
    would contradict the assumption of the lemma.

    Hence, by the induction hypothesis,
    there are $j \in [c]$ and a set $X' \subseteq V(G')$ of size at most
    $\sum_{i \in [c]} (|V(T_i)| + w) x'_i = \left(\sum_{i \in [c]}(|V(T_i)| + w) x_i\right) - (|V(T_{i_0})|+w)$,
    such that $x'_j \geq 1$ and there is no model of $(T_j,R_j)$ in $(G'-X', (S \cap V(G')) \setminus X)$.
    Then, we set
    \[
        X = X' \cup (A \cap B).
    \]
    First observe that $|X| \leq \sum_{i \in [c]} (|V(T_i)| + w) x_i$.
    It remains to show that $(T_j,R_j)$ is not a minor of $(G-X,S)$.
    Because $(A,B)$ is a separation of $G$ and because $T_j$ is connected,
    every model of $(T_j,R_j)$ in $(G-(A \cap B), S)$ is either included in $A \setminus B$ or in $B \setminus A$.
    Since $X'$ intersects every model of $(T_j,R_j)$ of $(G,S)$ included in $V(G') = B \setminus A$, 
    and since there is no model of $(T_j,R_j)$ in $(G,S)$ included in $A \setminus B$ by \Cref{claim:bd_tw},
    we deduce that $X = X' \cup (A \cup B)$ meets every model of $(T_j,R_j)$ in $(G,S)$.
    This concludes the proof of the lemma.
\end{proof}

We can now prove \Cref{thm:main_arbitrarily_rooted_forest}.

\begin{proof}[Proof of \Cref{thm:main_arbitrarily_rooted_forest}]
    We set, for every $t \in \NN$,
    \[
        g'(t) = g(t) + 2t^2.
    \]
    Let $F$ be a forest on $t$ vertices, and let $R \subseteq V(F)$.
    Let $T$ be a $t$-vertex tree such that $F$ is a subgraph of $T$.
    Let $G$ be a graph and $S \subseteq V(G)$,
    and let $k$ be a positive integer
    such that there are no $k$ pairwise vertex-disjoint models of $(F,R)$ in $(G,S)$.
    In particular, there are no $k$ pairwise vertex-disjoint models of $(T,V(T))$ in $(G,S)$.
    Hence, by \Cref{thm:main},
    there is a set $X_0 \subseteq V(G)$ of size 
    at most $g(t)k$ intersecting every model of $(T,V(T))$ in $(G,S)$.
    In other words, there is no $S$-rooted model of $T$ in $G-X_0$.
    By \Cref{diestelSrootedoriginal}, 
    $\tw(G-X_0, S) \leq \pw(G-X_0, S) < 2|V(T)|-1$.
    Then, by \Cref{lemma:bounded_tw_implies_EP}
    applied to $w=2|V(T)|-1$, $T_1, \dots, T_c$ being the connected components of $F$, and
    $x_i=k$ for every $i \in [c]$,
    there is a set $X_1 \subseteq V(G-Z_0)$ with
    \[
        |X_1| \leq \textstyle\sum_{i \in [c]}(|V(T_i)|+w) \cdot k
        = (t + wc) \cdot k 
        \leq 2 t^2 \cdot k
    \]
    that intersects every model of $(F,R)$ in $(G-X_0,S)$.
    The result then follows for $X = X_0 \cup X_1$.
\end{proof}

\section{Better bounds for paths and  stars}
\label{sec:paths_and_stars}

In this section, we establish better bounds for the function $g$ in \cref{thm:main} in the case of some specific trees $T$, namely when $T$ is  a path or a star. 
For $\ell \geq 1$, we denote by $P_\ell$ the path on $\ell$ vertices, and by $K_{1,\ell}$ the star with $\ell$ branches. 

\subsection{Paths}

We show that \cref{thm:main} holds with a function $g$ that is quadratic when the tree $T$ under consideration is a path. 

\begin{theorem}
    Let $G$ be a graph, let $S \subseteq V(G)$,
    and let $k,\ell$ be positive integers with $\ell \geq 2$.
    If there are no $k$ pairwise vertex-disjoint $S$-rooted models of $P_\ell$ in $G$,
    then there is a set $X \subseteq V(G)$ with $|X| \leq (\ell^2-1)(k-1)$
    such that there is no $S$-rooted model of $P_\ell$ in $G-X$. 
\end{theorem}

\begin{proof}
    Let $\mathcal{H}$ be a maximum-size family of pairwise vertex-disjoint subgraphs of $G$ 
    such that, for every $H \in \mathcal{H}$,
    $H$ contains an $S$-rooted model of $P_\ell$, and 
    suppose that $|\mathcal{H}| \leq k-1$. 
    By taking every member $H$ of $\mathcal{H}$ with $V(H),E(H)$ inclusion-wise minimal,
    we have that $H$ is the union of a path $Q_H$ with both endpoints in $S$,
    with $\ell-2$ vertex-disjoint $V(Q_H)$--$S$ paths $L_{H,1}, \dots, L_{H,\ell-2}$ that avoid the endpoints of $Q_H$.  
    Note that these paths may consist of a single vertex, which is then an interior vertex of $Q_H$ that is in $S$. 
    Let $Z_0$ be the set of all the vertices in $V(G)$ 
    that are endpoints of $Q_H$ for some $H \in \mathcal{H}$,
    or endpoints of $L_{H,j}$ for some $H \in \mathcal{H}, j \in [\ell-2]$.
    Note that
    \[
        |Z_0| \leq (\ell+(\ell-2))|\mathcal{H}| \leq (2\ell-2)(k-1).
    \]
    Let $\mathcal{P}$ be a maximum-size collection of pairwise vertex-disjoint $(V(\bigcup \mathcal{H})\setminus Z_0)$--$S$ paths in $G-Z_0$. 
   
    \begin{claim}\label{claim:Pl}
        For every $H \in \mathcal{H}$,
        \begin{enumerate}[label={\normalfont(\alph*)}]
            \item at most $\ell-1$ paths in $\mathcal{P}$ have an endpoint in 
                $Q_H$, and \label{item:claim:Pl:spine}
            \item for every $j \in [\ell-2]$, at most $\ell-1$ paths in $\mathcal{P}$ have an endpoint in $L_{H,j}$. \label{item:claim:Pl:legs}
        \end{enumerate}
    \end{claim}

    \begin{figure}
        \centering
        \includegraphics{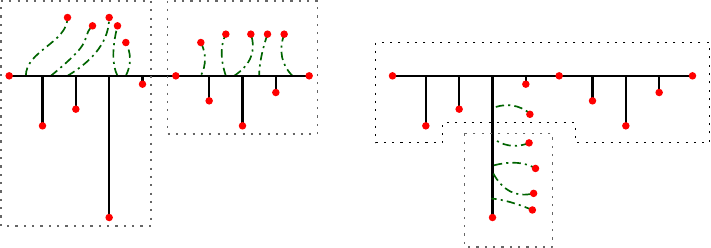}
        \caption{Illustration of \Cref{claim:Pl} for $\ell=10$.
            Case~\ref{item:claim:Pl:spine} is depicted on the left and Case~\ref{item:claim:Pl:legs} on the right.
            The graph $H \in \mathcal{H}$ is in solid black, the vertices of $S$ are red, and the paths in $\mathcal{P}$ are dashed green. In both cases, we find two vertex-disjoint $S$-rooted models of $P_\ell$, which contradicts the maximality of $|\mathcal{H}|$.}
        \label{fig:claim_Pl}
    \end{figure}

    \begin{proofclaim}
        See \Cref{fig:claim_Pl} for an illustration of this proof.
        Let $H \in \mathcal{H}$.
        We first prove \ref{item:claim:Pl:spine}.
        Suppose for contradiction that there are $\ell$ paths in $\mathcal{P}$
        with an endpoint in $Q_H$.
        Since these paths are vertex-disjoint from $Z_0$,
        we deduce that $H$ together with these paths contain an $S$-rooted model
        of $P_{2\ell}$.
        In particular, there are two vertex-disjoint $S$-rooted models of $P_\ell$ disjoint
        from $\mathcal{H} \setminus \{H\}$ (delimited by the dotted boxes in the figure), which contradicts the maximality of 
        $|\mathcal{H}|$.

        We now prove \ref{item:claim:Pl:legs}.
        Suppose for contradiction that there is a $j\in [\ell-2]$
        such that $L_{H,j}$ contains $\ell$ endpoints of paths in $\mathcal{P}$.
        Then again, the union of $H$ with these paths contains two vertex-disjoint 
        $S$-rooted models of $P_\ell$,
        which contradicts the maximality of $|\mathcal{H}|$.
    \end{proofclaim}

    It follows from \Cref{claim:Pl} that
    \[
        |\mathcal{P}| \leq \big((\ell-1) + (\ell-2)(\ell-1)\big)(k-1) = (\ell-1)^2(k-1).
    \]
    By Menger's Theorem, there is a set $Z_1$ of at most $(\ell-1)^2(k-1)$ vertices
    that intersects every
    $(V(\bigcup \mathcal{H})\setminus Z_0)$--$S$ paths in $G-Z_0$.
    Finally, we take
    \[
        X = Z_0 \cup Z_1.
    \]
    First, note that
    \[
        |X| \leq |Z_0| + |Z_1| \leq \big(2\ell-2 + (\ell-1)^2\big)(k-1) = (\ell^2 -1)(k-1),
    \]
    so it only remain to show that $G-X$ has no $S$-rooted model of $P_\ell$.
    Suppose for contradiction that such a model $\mathcal{M}$ exists.
    Then, by maximality of $|\mathcal{H}|$, $\mathcal{M}$ intersects some $H \in \mathcal{H}$.
    Moreover, $\mathcal{M}$ contains a vertex in $S$, and so we deduce that there is
    a path from $V(H) \setminus X$ to $S \setminus X$ in $G-X$,
    which contradicts the definition of $Z_1 \subseteq X$.
    This proves that $G-X$ has no $S$-rooted model of $P_\ell$
    and completes the proof.
\end{proof}

\subsection{Stars}
We show that \cref{thm:main} holds with a function $g$ that is linear when the tree $T$ under consideration is a star. 
To do so, we first need to introduce a couple of definitions and a lemma.

Let $G$ be a graph, and let $S\subseteq V(G)$.  
A tree $T$ that is a subgraph of $G$ is said to be \defin{$S$-good} if $T$ is subcubic, and every leaf of $T$ is in $S$.

\begin{lemma} \label{lem:splitting_a_tree}
  Let $G$ be a graph, let $S\subset V(G)$ and let $\ell$ be an integer with $\ell\geq 2$. Let $T$ be an $S$-good subtree of $G$ with at least $3\ell +1$ leaves. Then, there exist two vertex-disjoint $S$-good subtrees $T_1, T_2$ of $T$ such that $|L(T_1)|+|L(T_2)|\geq |L(T)|$ and $|L(T_1)|,|L(T_2)| \geq \ell+1$. 
\end{lemma}
\begin{proof}
    First observe that it suffices to show that there exist two vertex-disjoint subtrees $T'_1$ and $T'_2$ of $T$ such that both $T'_1$ and $T'_2$ have at least $\ell +1$ leaves in $L(T)$.  Indeed, $T_1, T_2$ can then be obtained by trimming respectively $T'_1, T'_2$ in such a way that all their leaves are in $L(T)$, and thus in $S$.

    Suppose for contradiction that there are no such two vertex-disjoint subtrees of $T$.  
    Then, by the Helly property of subtrees, 
    there is a vertex $z \in V(T)$ that hits all subtrees of $T$ that have at least $\ell +1$ leaves in $L(T)$. 
    In particular, every connected component of $T-z$ has at most $\ell$ leaves that are in $L(T)$. 
    If $z$ is a leaf of $T$, then there is only one connected component of $T-z$,
    and we deduce that $T$ has at most $\ell+1<3\ell+1$ leaves, a contradiction.
    If $z$ is not a leaf, then, since $z$ has degree at most three,
    we deduce that $T$ has at most $3\ell$ leaves, again a contradiction.
\end{proof}

We may now turn to our result for stars. 

\begin{theorem}
    Let $G$ be a graph, let $S\subset V(G)$, and let $k, \ell$ be positive integers.  Then, either $G$ contains $k$ vertex-disjoint $S$-rooted model of $K_{1,\ell}$, or there exists a set $X$ of at most $21\ell\cdot(k-1)$ vertices of $G$ such that $G-X$ has not $S$-rooted of $K_{1, \ell}$. 
\end{theorem}
\begin{proof}
    Assume that $G$ does not contain $k$ vertex-disjoint $S$-rooted models of $K_{1, \ell}$. 
    We show how to construct a set $X$ of at most $21\ell \cdot (k-1)$ vertices 
    such that $G-X$ has no $S$-rooted model of $K_{1, \ell}$.

    Let $\mathcal T=\{T_1, T_2, \ldots, T_{k'}\}$ be a (possibly empty) collection of $S$-good trees, such that, for $i\in [k']$, the number of leaves $\ell_i$ of $T_i$ is such that $\ell+1 \leq \ell_i \leq 3\ell$. 
    Moreover, we choose $\mathcal T$ such that $\sum_{i\in [k']} \ell_i$ is maximal. 

    For every $i \in [k']$, if $x_1, \dots,x_{\ell+1}$ are distinct leaves of $T_i$,
    then $(\{x_1\}, \dots, \{x_\ell\}, V(T_i) \setminus \{x_1, \dots, x_\ell\})$ is an $S$-rooted model of $K_{1,\ell}$ in $T_i$.
    Therefore, by assumption, we have $k'<k$.

    For each $i \in [k']$, let $X'_i$ be the set containing the vertices of degree $3$ of $T_i$ and the leaves of $T_i$.  
    Since $T_i$ has at most $3\ell$ leaves, we have $|X'_i|\leq 6\ell - 2$.  
    Let $X'= \bigcup_{i=1}^{k'} X'_i$.
    We have 
    \[
        |X'|\leq (6\ell-2)k' \leq (6\ell-2)(k-1).
    \]

    \begin{claim}\label{claim:Sl:X'}
        There is no path from $S \setminus V\left(\textstyle\bigcup \mathcal{T}\right)$ to $V(T_i)$ in $G-X'$, for every $i \in [k']$.
    \end{claim}
    \begin{proofclaim}
    Arguing by contradiction, suppose that there is a $\left(S\setminus (V(\bigcup \mathcal T)\right )$--$V(T_i)$ path $P$ in $G-X'$ for some $i \in [k']$.   
    Then, we can combine $P$ with the tree $T_i$ to obtain an $S$-good tree $T'_i$ with $\ell_i+1$ leaves that is vertex-disjoint from every tree in $\mathcal T\setminus \{T_i\}$.  
    If $\ell_i+1\leq 3\ell$, then $\mathcal{T}' = (\mathcal{T}\setminus \{T_i\}) \cup \{T'_i\}$ contradicts the maximality of $\sum_{i \in [k']} \ell_i$. 
    If $\ell_i+1>3\ell$, then $\ell_i+1=3\ell+1$.
    We apply \cref{lem:splitting_a_tree} to obtain two vertex-disjoint $S$-good subtrees $T''_i, T'''_i$ of $T'_i$ having between $\ell+1$ and $3\ell$ leaves, and whose total number of leaves is at least $|L(T'_i)|=\ell_i+1$.
    Then, $\mathcal{T}' = (\mathcal{T} \setminus \{T_i\}) \cup \{T''_i, T'''_i\}$ contradicts the maximality of $\sum_{i \in [k']} \ell_i$.
    \end{proofclaim}
    
    Let $j \in [k']$,  
    and let $\mathcal P$ be a maximum-size collection of vertex-disjoint $S\cap V(T_j))$--$(V(\bigcup \mathcal T)\setminus V(T_j))$ paths in $G-X'$.
    (Note that $\mathcal P$ could possibly be empty.) 
    It is clear, by construction, that $G[(V(\bigcup \mathcal T)\setminus V(T_j))\cup (\bigcup_{P\in \mathcal P}V(P))]$ contains a model of a forest consisting only of $S$-good trees, whose total number of leaves is $(\sum_{i\in [k']}\ell_i)-\ell_j + |\mathcal P|$.  By repeatedly applying \cref{lem:splitting_a_tree}, we obtain a forest consisting only of $S$-good trees with between $\ell+1$ and $3\ell$ leaves, without diminishing the number of leaves. 
    If $|\mathcal{P}|>\ell_j$, then $(\sum_{i \in [k']}\ell_i)-\ell_j + |\mathcal P| > \sum_{i\in [k']}\ell_i$, contradicting the maximality of $\sum_{i\in [k']}\ell_i$.  
    This proves that $|\mathcal{P}|\leq \ell_j$.
    By Menger's theorem, there exists a set $X''_j \subset V(G)\setminus X'$ of size at most $\ell_j$ separating $S\cap V(T_j)$ and $V(\bigcup \mathcal T)\setminus V(T_j)$ in $G-X'$.

    Let $X''= \bigcup_{i=1}^{k'} X''_i$.  
    Then 
    \[
    |X''|\leq \textstyle\sum_{i\in [k']} \ell_i \leq  3\ell (k-1), 
    \]
    and
    \begin{equation}\label{eq:Sl:X''}
        \begin{array}{c}
        \text{for all $i \in [k']$, there is no path}\\\text{from $\big (S\cap V(T_i) \big )$ to $\big(V(\bigcup \mathcal T) \setminus V(T_i)\big)$ in $G-(X' \cup X'')$.}
        \end{array}
    \end{equation}


    Now, for the remainder of the proof, we fix an arbitrarily chosen cubic tree $T$ that has exactly $3\ell+1$ leaves (and thus $6\ell$ vertices).

    \begin{claim}\label{claim:no_S_rooted_model_of_T}
       There is no $S$-rooted model of $T$ in $G-(X'\cup X'')$. 
    \end{claim}
    \begin{proofclaim}   
    Towards a contradiction, assume that there is an $S$-rooted model of $T$ in $G - (X' \cup X'')$, and let $G'$ be a connected component of $G$ containing such a model.
    Observe that $G'$ contains an $S$-good tree with at least $3\ell+1$ leaves, and thus by \cref{lem:splitting_a_tree}, it contains  two vertex-disjoint $S$-good trees $T', T''$ each having between $\ell+1$ and $3\ell$ leaves, and having $3\ell +1$ leaves in total. 
    Therefore, if $V(G')\cap V(\bigcup \mathcal T)=\emptyset$, we could extend $\mathcal T$, a contradiction.    
    Thus, $G'$ contains at least one vertex of $V(\bigcup \mathcal T)\setminus (X \cup X'')$. 
    Moreover, $G'$ also contains a vertex in $S$, which we denote by $s \in S \cap V(G')$.  
    If $s\not \in V(\bigcup \mathcal T)$, then \cref{claim:Sl:X'} implies that $G'$ intersects $X'$, a contradiction.
    Thus, $s\in V(T_i)$ for some $i$.
    By \eqref{eq:Sl:X''}, $V(G')$ is disjoint
    from $V(\bigcup \mathcal T)\setminus V(T_i)$.  
    Hence the set $(\mathcal T\setminus \{T_i\})\cup \{T', T''\}$ contradicts the maximality of $\sum_{i \in [k']} \ell_i$.  
    This proves that there is no $S$-rooted model of $T$ in $G - (X' \cup X'')$. 
    \end{proofclaim}   

    Combining \cref{claim:no_S_rooted_model_of_T} with \Cref{diestelSrootedoriginal}, we deduce that 
    \[
        \pw(G-(X'\cup X''), S)\leq 2|V(T)|-2=12\ell-2.
    \]
    Then, by \Cref{lem:erdos_posa_pw}, as there are no $k$ vertex-disjoint $S$-rooted models of $K_{1, \ell}$ in $G$ by assumption, there exists a set $X'''$ of size at most $(12\ell-1)(k-1)$ such that $(G-(X'\cup X''))-X'''$ does not contain any $S$-rooted minor of $K_{1, \ell}$.
    The result now follows by taking $X=X'\cup X''\cup X'''$, which has indeed size at most $(6\ell-2)(k-1)+(3\ell)(k-1)+(12\ell-1)(k-1)\leq 21\ell \cdot (k-1)$.
\end{proof}

\section{Open problems}
\label{sec:open_problems}

\Cref{thm:main} shows that there is
a linear Erd\H{o}s-P\'osa property for $S$-rooted models of a fixed tree $T$.
However, this linear factor depends heavily on $|V(T)|$.
It was conjectured by \citet{hodor2024quickly} that this dependency is actually linear in $|V(T)|$.
We conjecture the following explicit bound for this Erd\H{o}s-P\'osa function,
in the more general setting of $(F,R)$ models, where $R \subseteq V(F)$.

\begin{conjecture}
    Let $F$ be a forest and let $R \subseteq V(F)$.
    For every graph $G$ and every $S \subseteq V(G)$, for every positive integer $k$,
    at least one of the following holds:
    \begin{enumerate}[label={\normalfont(\arabic*)}]
        \item there are $k$ pairwise vertex-disjoint models of $(F,R)$ in $(G,S)$, or
        \item there is a set $X \subseteq V(G)$ with $|X| \leq |V(F)|(k-1)$
            such that there is no model of $(F,R)$ in $(G-X,S \setminus X)$.
    \end{enumerate}
\end{conjecture}

This bound would be tight for every $(F,R)$ and for every $k$:
the graph $G = K_{|V(F)|\cdot k-1}$ with $S = V(G)$ is such that there are no
$k$ pairwise vertex-disjoint models of $(F,R)$ in $(G,S)$, 
while no set of less than $|V(F)|(k-1)$ vertices intersects every model of $(F,R)$ in $(G,S)$.
Note also that the case where $F$ is a tree and $R = \emptyset$ 
was proved by \citet{DJMM25}.

\bibliographystyle{plainnat}
\bibliography{bibliography}

\end{document}